\documentclass{amsart}

\usepackage{amssymb,bbm}
\usepackage{amscd,amsfonts,amsmath,amssymb,bbm,dsfont}
\usepackage{enumerate,epsf,fancyhdr,float,graphicx,tabularx}
\usepackage{latexsym,mathrsfs,multirow}
\usepackage{wasysym}
\usepackage{xypic}
\usepackage[all]{xy}\usepackage[OT2,T1]{fontenc}
\usepackage[modulo,mathlines,displaymath, running]{lineno}
\usepackage{amsmath,mathrsfs,enumerate,wasysym}
\usepackage{xypic,hhline}
\usepackage[all]{xy}
\usepackage[OT2,T1]{fontenc}
\usepackage[modulo,mathlines,displaymath]{lineno}
\usepackage{tikz,tikz-cd}
\usepackage{dashrule}
\usetikzlibrary{matrix}
\usepackage[modulo,mathlines,displaymath]{lineno}

\newtheorem{theorem}{Theorem}[section]
\DeclareSymbolFont{cyrletters}{OT2}{wncyr}{m}{n}\DeclareMathSymbol{\Sha}{\mathalpha}{cyrletters}{"58}

\renewcommand{\phi}{{\varphi}}

\newcommand{\Fp}{\mathbf{F}_p}

\renewcommand{\geq}{\geqslant}

\newcommand{\smat}[1]{\left( \begin{smallmatrix} #1 \end{smallmatrix} \right)}

\newcommand{\links}{\left(\begin{array}{cc}}
\newcommand{\rechts}{\end{array}\right)}
\newcommand{\bai}{\left[\begin{array}{cc}}
\newcommand{\dai}{\end{array}\right]}
\newcommand{\hidari}{\left(\begin{array}{c}}
\newcommand{\migi}{\end{array}\right)}

\newcommand{\C}{\mathbb{C}}

\newcommand{\Q}{\mathbb{Q}}

\newcommand{\Z}{\mathbb{Z}}

\newcommand{\Log}{\mathcal Log}

\newcommand{\tensor}{\otimes} 

\renewcommand{\contentsname}{Contents\\{\footnotesize\normalfont(A table
of contents should normally not be included)}}

\newtheorem{auxiliary proposition}[theorem]{Auxiliary Proposition}

\newtheorem{corollary}[theorem]{Corollary}

\newtheorem{definition}[theorem]{Definition}

\newtheorem{lemma}[theorem]{Lemma}
\newtheorem{main conjecture}[theorem]{Main Conjecture}
\newtheorem{main theorem}[theorem]{Main Theorem}
\newtheorem{modesty proposition}[theorem]{Modesty Proposition}

\newtheorem{open problem}[theorem]{Open Problem}

\newtheorem{propn}[theorem]{Proposition}

\newtheorem{remark}[theorem]{Remark}

\newtheorem{convergence lemma}[theorem]{Convergence Lemma}
\newtheorem{corrected lemma}[theorem]{Corrected Lemma}
\newtheorem{growth lemma}[theorem]{Growth Lemma}
\newtheorem{coefficient lemma}[theorem]{Integrality Lemma}
\newtheorem{interpolation lemma}[theorem]{Interpolation Lemma}
\newtheorem{kernel lemma}[theorem]{Kernel Lemma}
\newtheorem{limit lemma}[theorem]{Limit Lemma}
\newtheorem{tandem lemma}[theorem]{Modesty Lemma}
\newtheorem{zero-finding lemma}[theorem]{Zero-Finding Lemma}

\definecolor{Green}{rgb}{0.0, 0.5, 0.0}

\begin{document}

\title{Consequences of functional equations for pairs of $p$-adic $L$-functions}

\begin{abstract}
We prove consequences of functional equations of $p$-adic $L$-functions for elliptic curves at supersingular primes $p$. The results include a relationship between the leading and sub-leading terms (for which we use ideas of Wuthrich and Bianchi), a parity result of orders of vanishing, and invariance of Iwasaswa invariants under conjugate 
twists of the $p$-adic $L$-functions.
\end{abstract}

\author{C\'{e}dric Dion}
\address{C\'{e}dric Dion,
D\'epartement de math\'ematiques et de statistique\\
Universit\'e Laval, Pavillon Alexandre-Vachon\\
1045 Avenue de la M\'edecine\\
Qu\'ebec, QC\\
Canada G1V 0A6}
\email{cedric.dion.1@ulaval.ca}

\author{Florian Sprung}
\address{Florian Sprung, 
 School of Mathematical and Statistical Sciences\\
Arizona State University\\
Tempe, AZ 85287-1804\\ USA}
\email{florian.sprung@asu.edu}

\maketitle

\section{Introduction}

In a beautifully short paper \cite{wuthrich}, Wuthrich related the leading term of the $L$-function of an elliptic curve $E$ with the second non-zero term in the Taylor expansion about $1$, which he called the sub-leading term. Bianchi adapted Wuthrich's ideas to $p$-adic $L$-functions of $E$. In case $p$ is a prime of ordinary reduction, Bianchi was able to go further. As a consequence of the nice behavior (integrality) of the $p$-adic $L$-function, Bianchi found a fact endemic to the $p$-adic world: 
The twist of the $p$-adic $L$-function by a 
 character $\psi$ has the same $\mu$-invariant as the twist by the conjugate character $\bar{\psi}$. 

A natural question is what happens when $p$ is a prime of supersingular (i.e. non-ordinary) reduction. The $p$-adic $L$-function for the ordinary case has two direct supersingular analogues. They are not power series with integral coefficients, so we can't speak of $\mu$-invariants. However, one can construct a pair $(L^\sharp,L^\flat)$ of integral $p$-adic $L$-functions \cite{ant}. The point of this article is thus to develop the ideas of Wuthrich and Bianchi in the setting of these $\sharp/\flat$ $L$-functions $L^{\sharp/\flat}$.

We relate the leading and sub-leading coefficients of $L^{\sharp/\flat}$ in the spirit of Bianchi and Wuthrich. Analogously to Bianchi, we deduce from this that the $\mu$-invariants of the twists of $L^{\sharp/\flat}$ by $\psi$ and $\bar{\psi}$ are the same. Do the $\lambda$-invariants also stay the same? We prove that the answer is yes. Further, we show that this $\lambda$-invariance holds also for the ordinary $p$-adic $L$-function.  Another new consequence is that the two non-integral $p$-adic $L$-functions in the supersingular case vanish to the same order modulo $2$ at the critical point. We can say the same about the functions $L^{\sharp/\flat}$.

The most important tool in the proofs is the functional equation. In \cite{ant}, two slightly different versions of pairs $(L^\sharp,L^\flat)$ were defined. We work with the first version, which is amenable to functional equations. The pair $(L^\sharp,L^\flat)$ we work with in this paper was denoted differently in \cite{ant} as $(\widehat{L}^\sharp,\widehat{L}^\flat)$, to record a completion that makes the functional equation work. But we drop the hats for notational convenience. The second non-completed version reduces to Pollack's $p$-adic $L$-functions $L_p^{\pm}$ defined when $a_p:=p+1-\#E(\Fp)=0$. The vanishing of $a_p$ guarantees that $L_p^{\pm}$ also fits a functional equation, allowing us to work out the results for $L_p^{\pm}$ as well.

Bianchi's idea for proving the invariance of the $\mu$-invariant is to show that the responsible term is in the same position for both power series in question. An alternative idea due to Pollack -- to simply observe that the two $p$-adic $L$-functions are Galois conjugates -- may be employed in the ordinary case and the case $a_p=0$ assuming further that $p \equiv 3 \mod 4$; his methods show further that the $\lambda$-invariants are the same in these cases as well.

By contrast, our methods count zeroes in the unit disk and work in all cases, for both the $\mu$ and the $\lambda$-invariants. The arguments may be of separate interest, and the reader may find them in the appendix.

{\textbf{Acknowledgment.} We would like to thank Antonio Lei for putting us in touch.}

\section{Notation}
Let $E$ be an elliptic curve over the field of rational numbers $\Q$ and let $p$ be a prime of good supersingular reduction for $E$. We denote by $\alpha$ and $\beta$ the two roots of the Hecke polynomial $X^2 -a_pX + p$. 
Amice and V\'{e}lu, and Vi\v{s}ik constructed two $p$-adic $L$-functions $L_{\alpha}(E,T)$ and $L_{\beta}(E,T)$ that each interpolate the special values of (cyclotomic twists of) the $L$-function of $E$. We denote their twists by Dirichlet characters $\psi$ by $L_{\alpha}(E,\psi,T)$ and $L_{\beta}(E,\psi,T)$, and make the convention of simply dropping the symbol $\psi$ when $\psi$ is the trivial character. See \cite[I, paragraph 13]{MTT}. 
We let $\overline{\psi}$ be the complex-conjugate of $\psi$. When the character of interrest is either trivial or quadratic, we denote it by $\chi$ to avoid confusion. $U$ denotes the open unit disk in $\C_p$, and we let $\Lambda=\Z_p[[T]]$. In the following section, we recall the main properties of the functions $\widehat{L}_p^{\sharp}(E,\psi,T)$ and $\widehat{L}_p^{\flat}(E,\psi,T)$ from \cite{ant}, which we denote without the hats as $L_p^{\sharp}(E,\psi,T)$ and $L_p^{\flat}(E,\psi,T)$ for convenience.

\section{Background}

In this section, we briefly recall the $p$-adic $L$-functions of Amice--V\'{e}lu and Vi\v{s}ik. See \cite{MTT} or the original papers \cite{vishik}, \cite{amicevelu}. We then recall the $\sharp/\flat$ $p$-adic $L$-functions of \cite{ant}. Their two essential properties are their relation with the Amice--V\'{e}lu--Vi\v{s}ik $p$-adic $L$-functions and their functional equation. In fact, we go a tiny bit beyond simply recalling them. While \cite{ant} only worked with twists by powers of the Teichmuller character, we state the results for the twists by any Dirichlet character $\psi$.

Amice--V\'{e}lu and Vi\v{s}ik's constructed $p$-adic distributions $\mu_{E,*}$ on $\Z_p^{\times}$ attached to $E$ for $* \in \{\alpha,\beta\}$. For $x \in \Z_p^{\times}$, denote by $\langle x \rangle$ the projection of $x$ on $1+p\Z_p$  (or on $1+4\Z_2$ when $p=2$). If $\psi$ is a Dirichlet character of conductor $p^kM$ with $p \nmid M$, their $p$-adic $L$-function $L_{*}(E,\psi,s)$ is then given by
$$
L_{*}(E,\psi,s) = \int_{\Z_{p,M}^{\times}} \psi(x)\langle x \rangle^{s-1} \mathrm{d}\mu_{E,*}.
$$
where $\Z_{p,M}^{\times} = \Z_p^{\times}\times (\Z/M\Z)^{\times}$. It is a $p$-adic locally analytic function in the variable $s\in \Z_p$. Let $\Gamma$ be the Galois group $\text{Gal}(\Q_{\infty}/\Q) \cong \Z_p$ of the cyclotomic $\Z_p$-extension of $\Q$. Choose $\gamma$ a topological generator of $\Gamma$ and let $\kappa$ be the cyclotomic character. The change of variables $T = \kappa(\gamma)^{s-1} -1$ transforms $L_{*}(E,\psi,s)$ into the power series
$$
L_{*}(E,\psi,T) :=\int_{\Z_{p,M}^{\times}}\psi(x) (1+T)^{\frac{\log \langle x \rangle}{\log \kappa(\gamma)}} \mathrm{d}\mu_{E,*}.
$$ From now on, all functions are in the variable $T$.

Choose $c_Q\in\{\pm1\}$ so that $f|_{\smat{0 & -1 \\ Q & 0}} = c_Q f$. Here $f$ is the modular form attached to $E$, and $Q$ is the largest divisor of the level $N$ of $E$ that is coprime to both $p$ and the conductor of $\psi$. The function $L_{*}(E,\psi,T)$ fits the functional equation
$$
L_{*}(E,\psi,T) = -(1+T)^{-\log_{\gamma}(Q)}\overline{\psi}(-Q)c_Q L_{*}(E,\overline{\psi}, \frac{1}{1+T}-1),
$$ \cite[Sections 5 and 17]{MTT}. We put $\log_{\gamma}(x) = \frac{\log \langle x \rangle}{\log \kappa(\gamma)}$ to ease notation.

Once $p$-adic numbers are identified with complex numbers, the special values of $L_{*}(E,T)$ at $T=\zeta_{p^n}-1$ can be related to special values of complex $L$-series. (More precisely, they are multiples by algebraic numbers normalized by a transcendental period $\Omega$ of the special values of twists of the Hasse-Weil $L$-functions at the complex value $1$ -- see \cite[Section 14]{MTT}.) We record the property for $T=0$:
$$
L_{*}(E,0) = \frac{1}{\Omega}\left( 1 - \frac{1}{*} \right)^2 L(E,1).
$$
Unlike in the ordinary case, the power series coefficients of $L_*(E,\psi,T)$ are unbounded so that $L_*(E,\psi,T) \notin \Lambda$ (and even $\notin\Lambda\tensor \C_p$). The main theorem of \cite{ant} remedies this:
\begin{theorem}
We have the factorization
$$
\left(L_{\alpha}(E,\psi,T),L_{\beta}(E,\psi,T)\right) = \left( L_p^{\sharp}(E,\psi,T),L_p^{\flat}(E,\psi,T) \right) \Log_{\alpha,\beta}(1+T)
$$
for two power series $L_p^{\sharp}(E,\psi,T),L_p^{\flat}(E,\psi,T)\in \Lambda$, where $\Log_{\alpha,\beta}(1+T)$ is an explicit $2 \times 2$ matrix of functions converging on $U$. 
\end{theorem}

\begin{proof}
This is \cite[Theorem 2.14]{ant}, where the statement is proved when $\psi$ is a power for the Teichmuller character. The same proof applies to arbitrary $\psi$. Note that $\Log_{\alpha,\beta}(1+T)$ was denoted $\widehat{\Log}_{\alpha,\beta}(1+T)$ in \cite[Section 4]{ant}. We don't recall its definition in this paper since it is not needed.
\end{proof}

 The $\sharp / \flat$ $L$-functions satisfy a functional equation similar to $L_{\alpha/\beta}(f,\psi,T)$. 

\begin{theorem}\label{functional}
Let $E$ be an elliptic curve over $\Q$ of level $N$ and $p$ a good supersingular prime. Let $\psi$ be a Dirichlet character, and let $Q$ be the largest divisor of $N$ that is coprime to both $p$ and the conductor of $\psi$. Then
\begin{align*}
L_p^{\sharp}(E,\psi,T) &= -(1+T)^{-\log_{\gamma}(Q)}\overline{\psi}(-Q)c_Q L_p^{\sharp}(E,\overline{\psi},\frac{1}{1+T}-1), \\
L_p^{\flat}(E,\psi,T) &= -(1+T)^{-\log_{\gamma}(Q)}\overline{\psi}(-Q)c_Q L_p^{\flat}(E,\overline{\psi},\frac{1}{1+T}-1).
\end{align*}
\end{theorem}

\begin{proof}
This follows from the same methods as in \cite[Theorem 5.19]{ant}. 
\end{proof}

\begin{remark}\label{completion} The original uncompleted $\sharp/\flat$ $p$-adic $L$-functions constructed in \cite{shuron} do not satisfy a nice functional equation, but there is an explicit relationship between these uncompleted and our completed $p$-adic $L$-functions, cf. \cite[Corollary 5.11]{ant}. \end{remark}

\section{The main results}
In this section, we state the main results of the paper. The first is the parity of the orders of vanishing at the critical point of the $\sharp/\flat$ $p$-adic $L$-functions, the second is the relation between the leading and sub-leading terms in each of them, and the third is the invariance of the Iwasawa invariants under the substitution $\psi \mapsto \overline{\psi}$, where $\psi$ is any Dirichlet character 

 Let $\chi$ be either the trivial character or a quadratic character. We let $Q$ and $N$ be as in the previous section: $Q$ denotes the largest divisor of the conductor $N$ of $E$ that is coprime to both $p$ and the conductor of $\chi$. Consider the Taylor expansions $$L_p^\sharp(E,\chi,T)=a_\sharp T^{m_\sharp}+b_\sharp T^{{m_\sharp}+1}+\cdots$$ and $$L_p^\flat(E,\chi,T)=a_\flat T^{m_\flat}+b_\flat T^{{m_\flat}+1}+\cdots,$$  where $m_{\sharp/\flat}$ denote the orders of vanishing at $T=0$ of $L^{\sharp/\flat}(E,\chi,T)$.

\begin{theorem}\label{vanishing}
	The order of vanishing of $L_p^{\sharp}(E,\chi,T)$ at $T=0$ has the same parity as the order of vanishing of $L_p^{\flat}(E,\chi,T)$ at $T=0$.
\end{theorem}

\begin{proof}
	Differentiating the functional equations in Theorem~\ref{functional} $m_{\sharp}$ resp. $m_{\flat}$ times and evaluating at $T=0$, we obtain $(-1)^{m_{\sharp}}=-c_Q$ and $(-1)^{m_{\flat}}=-c_Q$, cf. \cite[Proof of Theorem 4.1]{bianchi}. Thus, $m_{\sharp} \equiv m_{\flat} \mod 2$.
\end{proof}

\begin{corollary} The parities of the orders of vanishing of $L_\alpha(E,\chi,T)$ and $L_\beta(E,\chi,T)$ are the same.
\end{corollary}

\begin{proof}We obtain this by using the same arguments and the functional equations for $L_\alpha(E,\chi,T)$ and $L_\beta(E,\chi,T)$, cf. \cite[Section 17]{MTT}.\end{proof}

\begin{corollary} When $a_p=0$, the functions of Pollack $L_p^\pm(E,T)$ have the same parities of orders of vanishing.
\end{corollary}

\begin{proof} By \cite[Corollary 5.1]{ant}, $L_p^\pm(E,T)$ equal $L_p^{\sharp/\flat}(E,T)$ up to units.
\end{proof}

\begin{theorem}\label{main} We have $$b_\sharp=\frac{-a_\sharp}{2}\left(\log_{\gamma}(Q)+m_{\sharp}\right) \text{ and } b_\flat=\frac{-a_\flat}{2}\left(\log_{\gamma}(Q)+m_{\flat}\right)  .$$
\end{theorem}

\begin{proof} We differentiate the functional equations in Theorem~\ref{functional} $m_{\sharp}+1$ resp. $m_{\flat}+1$ times and then evaluate at $T=0$. This is analogous to \cite[Proof of Theorem 4.1]{bianchi}.
\end{proof}

We now let $\psi$ denote any Dirichlet character. 
\begin{definition} For the following theorem, we make the following conventions: We denote the $\lambda$-invariants of $L_p^{\sharp}(E,\psi,T)$ and $L_p^{\flat}(E,\psi,T)$ by $\lambda_\sharp$ and $\lambda_\flat$, and the $\lambda$-invariants of $L_p^{\sharp}(E,\overline{\psi},T)$ and $L_p^{\flat}(E,\overline{\psi},T)$ by $\overline{\lambda}_\sharp$ and $\overline{\lambda}_\flat$. 
We define the integers $\mu_\sharp,\mu_\flat,\overline{\mu}_\sharp,$ and $\overline{\mu}_\flat$ similarly.
\end{definition}
\begin{theorem}\label{invariance} The Iwasawa invariants of $L_p^{\sharp}(E,\psi,T)$ and $L_p^{\flat}(E,\psi,T)$ stay invariant under the substitution $\psi\mapsto \overline{\psi}$. More precisely, 
 \begin{enumerate}
 	\item We have $\lambda_\sharp=\overline{\lambda}_\sharp$ and $\lambda_\flat=\overline{\lambda}_\flat$.
 	\item We have $\mu_\sharp=\overline{\mu}_\sharp$ and $\mu_\flat=\overline{\mu}_\flat$.
 \end{enumerate}
\end{theorem}
\begin{proof}
	We carry out the proof for the $\sharp$ $p$-adic $L$-functions. By Theorem \ref{functional}, $L_p^{\sharp}(E,\psi,T)$ and $L_p^{\sharp}(E,\overline{\psi},\frac{1}{1+T}-1)$ differ by multiplication by an element of the form $-(1+T)^{-\log_{\gamma}(Q)}\overline{\psi}(-Q)c_Q$. By definition, $\overline{\psi}(-Q)c_Q$ is a unit. The term $-(1+T)^{-\log_{\gamma}(Q)}$ is a unit power series.
	
	 The $\lambda$-invariance and
		 the $\mu$-invariance are treated as general propositions in the appendix, see Proposition~\ref{lambda-invariance} and Proposition~\ref{mu-invariance}.
\end{proof}

\begin{corollary}The above theorem holds if we replace the $\sharp/\flat$ $p$-adic $L$-functions by their non-completed counterparts constructed in \cite{shuron}.
\end{corollary}
\begin{proof}Recall from Remark \ref{completion} that the non-completed $\sharp/\flat$ $p$-adic $L$-functions do not satisfy a direct functional equation. 
However, the $\lambda$- and $\mu$-invariants of the completed and non-completed $\sharp/\flat$ $p$-adic $L$-functions are the same, cf. \cite[Remark 2.11]{sha3}.
	\end{proof}

When $p$ is prime of good ordinary reduction for $E$, there is only one $p$-adic $L$-function, $L_p(E,\psi,T)$, coming from the unit root of $X^2 -a_pX + p$. Since proposition~\ref{lambda-invariance} and proposition~\ref{mu-invariance} are general statements about integral power series, they also apply to $L_p(E,\psi,T)$. Note that the invariance of $\mu$ was already known by Bianchi, albeit by using a different argument.

\begin{corollary}\label{invarianord} Let $p$ be a prime of ordinary reduction for $E$. Theorem \ref{invariance} also holds if we replace the $\sharp / \flat$ $p$-adic $L$-functions by their analogue in the ordinary case $L_p(E,\psi,T)$.
\end{corollary}
\begin{proof}
Since $L_p(E,\psi,T)$ is an integral power series and satisfies the same functional equation as $L_p^{\sharp / \flat}$, the proof of theorem \ref{invariance} can be used in this case as well.
\end{proof}

\begin{remark}[Pollack] One can obtain the invariance of Iwasawa invariants in the ordinary case (Corollary \ref{invarianord}) and in the subcase $a_p=0$ and $p\equiv 3 \mod 4$ of the supersingular case (Theorem \ref{invariance}) by simply observing that the $\psi$ and $\overline{\psi}$ are Galois-conjugate characters, cf. the arguments in \cite[Lemma 6.7]{pollack}.
	\end{remark}

\section{The case $a_p=0$}

When $a_p=0$ and $\psi$ is the trivial character, the functions $L_p^{\sharp / \flat}$ are Pollack's $L^{\pm}$ functions multiplied by a unit, cf. \cite[Corollary 5.1]{ant}. See \cite{pollack} for the original definition of $L^{\pm}$.
The functional equations for the plus/minus $L$-functions are \cite[Section 5.1]{ant}:
\begin{align}
	L_p^{+}(E,T) &= -(1+T)^{-\log_{\gamma}(N)}c_N W^{+}(1+T)L_p^{+}(E,\frac{1}{1+T}-1), \\
	L_p^{-}(E,T) &= -(1+T)^{-\log_{\gamma}(N)}c_N W^{-}(1+T)L_p^{-}(E,\frac{1}{1+T}-1),
\end{align}
where $W^{\pm}(1+T)$ are units given by
\begin{align*}
	W^{+}(1+T) &= \prod_{j \geq 1}(1+T)^{-p^{2j-1}(p-1)}, \\
	W^{-}(1+T) &=\begin{cases} \prod_{j \geq 1}(1+T)^{-p^{2j-2}(p-1)} &\text{for $p$ odd}, \\ (1+T)^{-1}\prod_{j \geq 2}(1+T)^{-p^{2j-2}(p-1)} &\text{for $p=2$}. \end{cases}
\end{align*}

\begin{theorem}
	Let $E$ be an elliptic curve over $\Q$ of conductor $N$ and $p$ a prime such that $(p,N)=1$ and $a_p=0$. Write
	\begin{align*}
		L_p^{+}(E,T) &= a_{m_+}T^{m_{+}} + b_{m_+}T^{1+m_{+}} + \cdots \\
		L_p^{-}(E,T) &= a_{m_-}T^{m_-} + b_{m_-}T^{1+m_-} + \cdots
	\end{align*}
	where $m_{\pm}$ denotes the order of vanishing of $L_p^{\pm}(E,T)$ at $T=0$. We then have
	$$
	b_{m+} = -\frac{a_{m_+}}{2} \left( \log_{\gamma}(N) -\frac{p}{1+p}  +m_+\right)
	\text{and }
	b_{m_-} =-\frac{a_{m_-}}{2} \left( \log_{\gamma}(N) - \frac{1}{1+p} + m_- \right). 
	$$
\end{theorem}

\begin{proof}
	We let $F^{\pm}(T) = (1+T)^{-\log_{\gamma}(N)}W^{\pm}(1+T)$. We follow the same arguments as in the proofs of Theorems \ref{vanishing} and \ref{main}. Differentiating $m_{+}$ times the functional equation (1) and $m_{-}$ times (2) on both sides and evaluating at $T=0$ gives $-c_N = (-1)^{m_{\pm}}$. Then, differentiating $m_{\pm} +1 $ times on both sides of the respective functional equation and evaluating at $T=0$ yields
	$$
	b_{m+} = \frac{-a_{m_+}}{2} \left( -(F^{+})^{\prime}(0) + m_+ \right)
	\text{ and }
	b_{m_-} = \frac{-a_{m_-}}{2} \left( -(F^{-})^{\prime}(0) + m_- \right).
	$$
	 But for odd $p$, $$(F^{+})^{\prime}(0) = -\log_{\gamma}(N) + \frac{p}{p+1} \text{ and } (F^{-})^{\prime}(0) = -\log_{\gamma}(N) + \frac{1}{p+1} \text{, while}$$ $$(F^{-})^{\prime}(0) = -\log_{\gamma}(N) + \frac{p^2-p-1}{p+1} \text{ for $p=2$}.$$ The coincidence $\frac{2^2-2-1}{2+1}=\frac{1}{2+1}$ shows that the result is the same for $p=2$!
\end{proof}

\newpage
\appendix
\section{Invariance of  Iwasawa invariants under functional equations}
This appendix contains results used in the proof of Theorem \ref{invariance}, which may be of general interest.
\begin{propn}[$\lambda$-invariance]\label{lambda-invariance}
	Let $f(T)$ and $g(T) \in \Lambda$ be related by
	$$g(T)=uf\left(\frac{1}{1+T}-1\right)$$
	for a unit $u\in \Lambda^\times$. Then we have that $\lambda(f(T))=\lambda(g(T))$.
\end{propn}
\begin{proof} We must prove that $f(T)$ and $g(T)$ have the same number of zeroes in $U$ (counted with multiplicity). Let $\zeta$ be such a zero for $f(T)$, i.e. let $(T-\zeta)$ be a factor of $f(T)$ in $\overline{\Z}_p[T]$. Then $-\frac{\zeta}{1+\zeta}$ is a zero for $g(T)$ in $U$. We want to show that the factor $(T-\zeta)$ appears in $f(T)$ as often as $\left(T+\frac{\zeta}{1+\zeta}\right)$ does in $g(T)$.
	
	For this, consider the term
	$$\left(\frac{1}{1+T}-1-\zeta\right)=U(T)T-\zeta, $$ where
	$U(T)=\frac{(1+T)^{-1}-1}{T}\in\Lambda^\times$. Our term vanishes at $T=-\frac{\zeta}{1+\zeta}$, so 
	$$U(T)T-\zeta=\left(T+\frac{\zeta}{1+\zeta}\right)\times V_\zeta(T)$$ 
	for some $V_\zeta(T)\in\overline{\Z}_p[T]$. We introduced $U(T)$ in this proof because it helps us evaluate $V_\zeta(T)$ at $T=0$: $V_\zeta(0)=-(1+\zeta)\in\overline{\Z}_p^\times$, so $V_\zeta(T)$ is a unit. 
	
	 Under $T\mapsto \frac{1}{1+T}-1$, we thus have $(T-\zeta)^m \mapsto \left (T+\frac{\zeta}{1+\zeta}\right)^mV_\zeta^m(T)$. Thus,
	$$\{\text{zeros of $f(T)$ of multiplicity $n$ in $U$}\}\subset \{\text{zeros of $g(T)$ of multiplicity $\geq n$ in $U$}\}.$$
	Carrying out the argument with $f$ and $g$ reversed gives us a bijection between zeroes of $f$ and $g$ in $U$ that respects their multiplicities.
\end{proof}

\begin{lemma}[Unit invariance] Let $U(T)\in \Lambda^\times$. Then $U(\frac{1}{1+T}-1)\in \Lambda^\times$. \end{lemma}
\begin{proof} Evaluation at $T=0$.
\end{proof}
\begin{propn}[$\mu$-invariance]\label{mu-invariance}
	Let $f(T)$ and $g(T) \in \Lambda$ be related by
	$$g(T)=uf\left(\frac{1}{1+T}-1\right)$$
	for a unit $u\in \Lambda^\times$. Then we have that $\mu(f(T))=\mu(g(T))$.
\end{propn}
\begin{proof} Write $f(T)=p^\mu D(T)U(T)$ for an integer $\mu$, $D(T)$ distinguished, and $U(T)\in \Lambda^\times$. In $\overline{\Z}_p[[T]]$, we factor $f(T)=p^\mu\left(\prod_\zeta(T-\zeta)\right) U(T)$. From the proof of the $\lambda$-invariance lemma, we have $g(T)=up^\mu\left(\prod_\zeta(T+\frac{\zeta}{1+\zeta})\prod_\zeta V_\zeta(T)\right) U(\frac{1}{1+T}-1)$. The term $\prod_\zeta \left(T+\frac{\zeta}{1+\zeta}\right)V_\zeta(T)$ is in $\Lambda$ and has $\mu$-invariant zero, while $U(\frac{1}{1+T}-1)$ is a unit by the unit-invariance lemma above. The result follows.
\end{proof}

\bibliography{ssBianchi}{}
\bibliographystyle{plain}

\end{document}